\documentclass[10pt,a4paper,twoside]{article}
\usepackage{amssymb}
\usepackage{amsmath}
\usepackage{fancyhdr}
\usepackage{amsthm}
\usepackage[T1]{fontenc}

\usepackage{afterpage}

 \newtheorem{thm}{Theorem}[section]
 
 \newtheorem{lem}[thm]{Lemma}
 \newtheorem{prop}[thm]{Proposition}
 \theoremstyle{definition}
 \newtheorem{defn}[thm]{Definition}
 \theoremstyle{remark}

 \numberwithin{equation}{section}

\newcommand{\eqn}[2]{\begin{equation} \label{eq:#2} #1 \end{equation}}

\newcommand{\qqqquad}{\qquad \qquad \qquad \qquad}

\newcommand{\ld}{\backslash} 
\newcommand{\rd}{/} 
\newcommand{\ldd}{\backslash\! \!\backslash} 
\newcommand{\rdd}{/ \! \! /} 

\newcommand{\gD}{\Delta}

\newcommand{\op}[1]{\bigcirc\!\!\!\!\!\!{#1}\;\:}
\def\nn1{\operatorname{NewName'}}

\newcommand\mj{\mbox{\bf 1}}
\newcommand\str{\rightarrow}
\newcommand\strt{\stackrel{\textbf{.}\,}{\rightarrow}}

\begin{document}

\setcounter{page}{1}                                       

\begin{center}
\vspace*{2pt}
{\Large \textbf{A note on semisymmetry}}\\[3mm]
{\large \textsf{\emph{Aleksandar Krape\v{z} and Zoran Petri\'c}}}
\\[30pt]
\end{center}
\textbf{Abstract.} \\ A survey of properties of the adjunction
involving a semisymmetrization functor, which was suggested by
J.D.H.\ Smith, and which maps the category of quasigroups with
homotopies to the category of semisymmetric quasigroups with
homomorphisms, is given. A new semisymmetrization functor is
suggested. This functor maps a quasigroup to its square instead to
its cube as it was the case with the former functor.
 \vspace*{5mm}

\section{Introduction}

For a plausible category of quasigroups, it seems that homotopies
between quasigroups, taken as morphisms, are better choice than
homomorphisms (see \cite{Gvaramiya}). However, homomorphisms are
sometimes easier to work with. For example, isotopies do not
preserve units---every quasigroup is isotopic to a loop
(quasigroup with a unit) but is not necessarily a loop itself.
This note is about turning homotopies into homomorphisms.

Smith, \cite{Smith}, proved that there is an adjunction from the
category of semisymmetric quasigroups with homomorphisms to the
category of quasigroups with homotopies. Also, he proved in
\cite{Smith} that the latter category is isomorphic to a
subcategory of the former category, and in \cite{Smith8}, that
every $T$ algebra, for $T$ being the monad defined by the above
adjunction, is isomorphic to the image of a semisymmetric
quasigroup under the comparison functor.

These results, especially the embedding of the category of
quasigroups with homotopies into the category of semisymmetric
quasigroups with homomorphisms, could be of interest to a working
universal algebraist. Our intention is to make them more
accessible to such a reader and to indicate a possible misusing.
Also, we make some alternative proofs and add the fullness of the
comparison functor in order to complete the picture of this
adjoint situation.

At the end of the paper, we show that there is a more economical
way to embed the category of quasigroups with homotopies into the
category of semisymmetric quasigroups with homomorphisms. One
could get an impression, due to \cite{Smith}, that for such an
embedding it is necessary to have a semisymmetrization functor
that is a right adjoint in an adjunction. If one is interested
just in this embedding and not in reflectivity (see the end of
Section \ref{adj}), then this new semisymmetrization suits as any
other.

We assume that the reader is familiar with the notions of
category, functor and natural transformation. If not, we suggest
to consult \cite{ML71} for these notions. All other relevant
notions from Category theory are introduced at the appropriate
places in the text.

\section{Quasigroups}\label{Q}

We start by
recapitulating a few basic facts about quasigroups. 

One way to define a \emph{quasigroup} is that it is a grupoid $(Q;
\cdot)$ satisfying:
\begin{equation}
\forall{a b} \exists_{1}x (x \cdot a = b) \qquad \quad \text{ and
} \qquad \quad \forall{a b} {\exists}_1 x (a \cdot x = b) \notag
\end{equation}
Uniqueness of the solution of the equation $x\cdot a = b \;
(a\cdot x = b)$ enables one to define {right (left) division}
operation $x = b \rd a \; (x = a \ld b)$ which is also a
quasigroup (short for: $(Q;\rd)$ is a quasigroup). We can define
three more operations:
$$
x * y = y\cdot x  \hspace{2cm} x \rdd y = y \rd x  \hspace{2cm} x
\ldd y = y \ld x
$$
dual to $\cdot, \rd, \ld$\/ respectively. They are also
quasigroups.
%
%
%
The six operations $\cdot, \rd, \ld, *, \rdd$ and $\ldd$ are
\emph{parastrophes} of $\cdot$ (and of each other).

A function $f : Q \to R$  between the base sets of quasigroups
$(Q; \cdot)$ and $(R, \cdot)$ is a \emph{homomorphism} iff:
$$f(x) \cdot f(y) = f(x\cdot y)$$
and \emph{isomorphism} if $f$ is a bijection as well.

A triple $\bar{f} = (f_1, f_2, f_3)$ of functions ($f_i : Q \to
R$) is a {homotopy} iff:
$$f_1(x) \cdot f_2(y) = f_3(x\cdot y)$$
which implies (and is implied by any of):
\begin{alignat}{2}
f_3(x) \rd f_2(y) &= f_1(x \rd y)     & \qqqquad f_2(x) \rdd f_3(y) &= f_1(x \rdd y) \notag \\
f_1(x) \ld f_3(y) &= f_2(x \ld y) & \qqqquad f_3(x) \ldd f_1(y) &=
f_2(x \ldd y) \notag
\end{alignat}
If all three components of $\bar{f}$ are bijections, then
$\bar{f}$ is an {isotopy}.

\begin{center}
***
\end{center}

We can also
define a quasigroup as an algebra $(Q; \cdot, \rd, \ld)$ with three binary operations: multiplication ($\cdot$), right and left division.
The axioms that a quasigroup satisfies are ($xy$ is short for $(x \cdot y)$): \\
\parbox{8cm}{\[
\begin{alignedat}{2}
xy \rd y &= x     & \qqqquad x \ld xy &= y  \\
(x \rd y) y &= x  & \qqqquad x (x \ld y) &= y
\end{alignedat}\]}
\hfill
\parbox{3cm}
{
\begin{equation*}
\tag{Q}
\end{equation*}}

\noindent For obvious reasons, such quasigroups are called
\emph{equational, primitive} or \emph{equasigroups}.

Thus, we have the variety of all quasigroups. Another important
variety is the variety of semisymmetric quasigroups, defined by
one of the following five equivalent axioms (in addition to (Q)):
\eqn{x \cdot yx = y}{SS1} \eqn{xy \cdot x = y}{SS2}
\[
x \rd y =yx
\]
\[
x \ld y = yx
\]
\[
x \ld y = x \rd y
\]

Smith, \cite{Smith}, defined a \emph{semisymmetrization} of a
quasigroup $\mathbb{Q} = (Q; \cdot, \rd, \ld)$ as a one--operation
quasigroup $\mathbb{Q}^{\gD} = (Q^3; \circ)$ where the binary
operation $\circ$ is defined by: \eqn{(x_1,x_2,x_3) \circ
(y_1,y_2,y_3) = (y_3 \rd x_2, y_1 \ld x_3, x_1 y_2)}{SSprod} and
proved that, for any quasigroup $\mathbb{Q}$\/, the
semisymmetrization $\mathbb{Q}^{\gD}$ of $\mathbb{Q}$ is a
semisymmetric quasigroup.
%

\section{Twisted quasigroups}\label{twist}

For our purpose, there is a better way to define a quasigroup. In
this definition the \emph{twisted quasigroup} is an algebra $(Q;
\rdd, \ldd, \cdot)$ satisfying appropriate paraphrasing of the
above quasigroup axioms (Q):
\begin{alignat}{2}
\label{ax2}
y \rdd xy &= x     & \qqqquad xy \ldd x &= y \notag \\
(y \rdd x) y &= x  & \qqqquad x (y \ldd x) &= y \notag
\end{alignat}

We have the following symmetry result, lacking for quasigroups
defined as $(Q; \cdot, \rd, \ld)$\/.
\begin{prop}
\label{TQ} An algebra $(Q; \rdd, \ldd, \cdot)$ is a \emph{twisted
quasigroup} iff $(Q; \ldd, \cdot, \rdd)$ is a \emph{twisted
quasigroup} iff $(Q; \cdot, \rdd, \ldd)$ is a \emph{twisted
quasigroup}.
\end{prop}

Analogously, we have the paraphrasing of axioms for \emph{twisted
semisymmetric quasigroups}: (\ref{eq:SS1}),(\ref{eq:SS2}) and
\[
x \rdd y = xy
\]
\[
x \ldd y = xy
\]
\[
x \ldd y = x \rdd y
\]
The last three identities we shorten to symbolic identities:
$\rdd = \cdot, \ldd = \cdot, \ldd = \rdd$\/.

There is also a result corresponding to Proposition~\ref{TQ}:
\begin{prop}
\label{SQ} An algebra $(Q; \rdd, \ldd, \cdot)$ is a
\emph{semisymmetric twisted quasigroup} iff $(Q; \ldd, \cdot,
\rdd)$ is a \emph{semisymmetric twisted quasigroup} iff $(Q;
\cdot, \rdd, \ldd)$ is a \emph{semisymmetric twisted quasigroup}.
\end{prop}

\begin{center}
***
\end{center}

Using twisted quasigroups we can see how a (twisted)
semisymmetrization (defined below), which we call $\nabla$\/,
'works'.

Let us start with three single--operation quasigroups $(Q; \cdot),
(Q; \rdd)$ and $(Q; \ldd)$\/, where $\rdd$ and $\ldd$ are duals of
appropriate division operations of $\cdot$\/. We can define direct
(Cartesian) product
 $(Q; \rdd) \times(Q; \ldd) \times (Q; \cdot)$ and an operation $\otimes$ on $Q^3$ such that
\eqn{(x_1,x_2,x_3) \otimes (y_1,y_2,y_3) = (x_1 \rdd y_1, x_2 \ldd
y_2, x_3 y_3)}{Xprod} defines multiplication in the direct
product. Therefore $(Q^3; \otimes)$ is a quasigroup.

Define also a permutation $' : Q^3 \to Q^3$ by $(x_1,x_2,x_3)' =
(x_2,x_3,x_1)$\/. It follows that $(x_1,x_2,x_3)'' =
(x_3,x_1,x_2)$ and $(x_1,x_2,x_3)''' = (x_1,x_2,x_3)$\/. Define
another operation $\nabla_3 : Q^3 \times Q^3\to Q^3$ by $\bar{x}
\nabla_3 \bar{y} = \bar{x}' \otimes \bar{y}''$, where
$\bar{u}=(u_1,u_2,u_3)$\/. The groupoid $(Q^3; \nabla_3)$ is also
a quasigroup, so there are appropriate division operations of
$\nabla_3$ and their duals $\nabla_1$ and $\nabla_2$:
$$\bar{x} \nabla_3 \bar{y} = \bar{z} \mbox{\quad  iff \quad}
\bar{y} \nabla_1 \bar{z} = \bar{x} \mbox{\quad iff \quad} \bar{z}
\nabla_2 \bar{x} = \bar{y} .$$ Therefore $(Q^3; \nabla_1,
\nabla_2, \nabla_3)$ is a twisted quasigroup.

Let us calculate $\nabla_1$\/.
\begin{tabbing}
\hspace{1.5em} $\bar{z}$ \= $ = (z_1,z_2,z_3) = \bar{x} \nabla_3
\bar{y} = (x_1,x_2,x_3)' \otimes (y_1,y_2,y_3)''$
\\[1ex]
\> $ = (x_2,x_3,x_1) \otimes (y_3,y_1,y_2) = (x_2 \rdd y_3, x_3
\ldd y_1, x_1 y_2)$\/.
\end{tabbing}
Therefore
\[
\bar{x} = (y_2 \rdd z_3, y_3 \ldd z_1, y_1 z_2) = (y_2,y_3,y_1)
\otimes (z_3,z_1,z_2) = \bar{y}' \otimes \bar{z}'' = \bar{y}
\nabla_3 \bar{z}
\]
i.e. $\nabla_1 = \nabla_3$ (and consequently $\nabla_2 =
\nabla_3$) hence $(Q^3; \nabla_1, \nabla_2, \nabla_3)$ is
semisymmetric twisted quasigroup. So we recognize $\nabla_3$ as a
twisted analogue of Smith's $\circ$ (see identity (\ref{eq:SSprod})). Let
us call $\mathbb{Q}^{\nabla} = (Q^3; \nabla_1, \nabla_2,
\nabla_3)$ a \emph{twisted semisymmetrization} of $\mathbb{Q}$\/.

For $(f_1,f_2,f_3)$ being a homotopy from $\mathbb{Q}$ to
$\mathbb{R}$, we also have:
\begin{tabbing}
\hspace{1.5em}$(f_1\times f_2\times f_3)\,(\bar{x} \nabla_3
\bar{y})$ \= $=(f_1\times f_2\times f_3)\,(\bar{x}' \otimes
\bar{y}'') $
\\[1ex]
\> $= (f_1(x_2 \rdd y_3), f_2(x_3 \ldd y_1), f_3(x_1 \cdot y_2))$
\\[1ex]
\> $= (f_2 x_2 \rdd f_3 y_3, f_3 x_3 \ldd f_1 y_1, f_1 x_1 \cdot
f_2 y_2)$
\\[1ex]
\> $= (f_2 x_2, f_3 x_3, f_1 x_1) \otimes (f_3 y_3, f_1 y_1, f_2
y_2)$
\\[1ex]
\> $  = (f_1 x_1, f_2 x_2, f_3 x_3)' \otimes (f_1 x_1, f_2 x_2,
f_3 x_3)''$
\\[1ex]
\> $ = (f_1\times f_2\times f_3)\,(\bar{x}) \nabla_3 (f_1\times
f_2\times f_3)\,(\bar{y})$\/,
\end{tabbing}
so $f_1\times f_2\times f_3$ is a homomorphism.

\section{Biquasigroups}\label{biq}

\begin{defn}
\label{bi} An algebra $(Q; \rdd, \ldd)$ is a \emph{biquasigroup} iff
$\rdd (\ldd)$ is the dual of the right (left) division operation
of a quasigroup operation $\cdot$\/.

A biquasigroup is semisymmetric iff $\ldd = \rdd$\/.
\end{defn}
\begin{prop}
An algebra $(Q; \rdd, \ldd)$ is a biquasigroup iff $(Q; \ldd,
\cdot)$ is a biquasigroup iff $(Q; \cdot, \ldd)$ is a
biquasigroup.
\end{prop}
\begin{prop}
An algebra $(Q; \rdd, \ldd)$ is a semisymmetric  biquasigroup iff $(Q; \ldd,
\cdot)$ is a semisymmetric biquasigroup iff $(Q; \cdot, \ldd)$ is a
semisymmetric biquasigroup.
\end{prop}
\begin{center}
***
\end{center}

Let us start with three single--operation quasigroups $(Q; \cdot),
(Q; \rdd)$ and $(Q; \ldd)$\/, where $\rdd$ and $\ldd$ are duals of
appropriate division operations of $\cdot$\/. We can define direct
(Cartesian) product
 $(Q; \rdd) \times(Q; \ldd)$ and an operation $\otimes$ on $Q^2$ such that
\[
(x_1,x_2) \otimes (y_1,y_2) = (x_1 \rdd y_1, x_2 \ldd
y_2)
\]
defines multiplication in the direct product. Therefore $(Q^2;
\otimes)$ is a quasigroup.

Define also a permutation $' : Q^2 \to Q^2$ by $(x_1,x_2)' =
(x_2,x_1)$\/.
Define another operation $\nabla : Q^2\times Q^2 \to Q^2$ by
$\hat{x} \nabla \hat{y} = R_{\hat{y}}(\hat{x}') \otimes
L_{\hat{x}}(\hat{y}')$\/, where $\hat{u}$ is $(u_1,u_2)$,
$L_{\hat{x}}(\hat{y}) = (x_1\cdot y_1, y_2)$ and
$R_{\hat{y}}(\hat{x}) = (x_1, x_2\cdot y_2)$\/. The groupoid
$(Q^2; \nabla)$ is also a quasigroup, moreover a semisymmetric
one.
Therefore $(Q^2; \nabla, \nabla)$ is a semisymmetric
biquasigroup.

Let us define:
$$x \op{1} y = x \rdd y \qquad\qquad
x \op{2} y = x \ldd y \qquad\qquad x \op{3} y = x \cdot y$$ Then
the definition of $\nabla_{12}$, which we abbreviate just by
$\nabla$, is:
$$(x_1,x_2) \nabla_{12} (y_1,y_2) = (x_2 \op{1} (x_1 \op{3} y_2), (x_1 \op{3} y_2) \op{2} y_1).$$
There are two more alternative semisymmetrizations with corresponding definitions in $(Q^2; \ldd, \cdot)$
(respectively $(Q^2; \cdot, \rdd)$):
$$(x_1,x_2) {\nabla}_{23} (y_1,y_2) = (x_2 \op{2} (x_1 \op{1} y_2), (x_1 \op{1} y_2) \op{3} y_1)$$
$$(x_1,x_2) {\nabla}_{31} (y_1,y_2) = (x_2 \op{3} (x_1 \op{2} y_2), (x_1 \op{2} y_2) \op{1} y_2).$$
The indexing of operations is used to emphasize the symmetry.
These semisymmetrizations are object functions of functors related
to a functor explained in details in Section~\ref{new}.

\section{The categories \textbf{Qtp} and \textbf{P}}\label{adj}

This section follows
the lines of \cite{Smith} with some adjustments. The main novelty
is a proof of \cite[Corollary~5.3]{Smith}. We try to keep to the
notation introduced in \cite{Smith}. However, we write functions
and functors to the left of their arguments.

Let \textbf{Qtp} be the category with objects all small quasigroups
$\mathbb{Q}=(Q; \cdot, \rd, \ld)$ and arrows all homotopies. The
identity homotopy on $\mathbb{Q}$ is the triple
$(\mj_Q,\mj_Q,\mj_Q)$, where $\mj_Q$ is the identity function on
$Q$, and the composition of homotopies
\[
(f_1,f_2,f_3)\!:\mathbb{P}\rightarrow \mathbb{Q}\quad\mbox{\rm
and}\quad (g_1,g_2,g_3)\!:\mathbb{Q}\rightarrow \mathbb{R}
\]
is the homotopy
\[
(g_1\circ f_1,g_2\circ f_2,g_3\circ f_3)\!:\mathbb{P}\rightarrow
\mathbb{R}.
\]

Let \textbf{P} be the category with objects all small
semisymmetric quasigroups and arrows all quasigroup homomorphisms.
For every arrow $f\!:\mathbb{Q}\str \mathbb{R}$ of \textbf{P}, the
triple $(f,f,f)$ is a homotopy between $\mathbb{Q}$ and
$\mathbb{R}$.

Let $\Sigma$ be a functor from \textbf{P} to \textbf{Qtp}, which
is identity on objects. Moreover, let $\Sigma f$, for a
homomorphism $f$, be the homotopy $(f,f,f)$.

The category \textbf{P} is a full subcategory of the category
\textbf{Q} with objects all small quasigroups and arrows all
quasigroup homomorphisms. The functor $\Sigma$ is just a
restriction of a functor from \textbf{Q} to \textbf{Qtp}, which is
defined in the same manner.

An \emph{adjunction} is given by two functors, $F\!:\textbf{C}\str
\textbf{D}$ and $G\!:\textbf{D}\str \textbf{C}$, and two natural
transformations, the \emph{unit} $\eta\!:\mj_\textbf{C}\strt GF$
and the \emph{counit} $\varepsilon\!:FG\strt \mj_\textbf{D}$, such
that for every object $C$ of $\textbf{C}$ and every object $D$ of
$\textbf{D}$
\[
G\varepsilon_D \circ \eta_{GD}=\mj_{GD}, \quad \mbox{\rm and}\quad
\varepsilon_{FC}\circ F\eta_C=\mj_{FC}.
\]
These two equalities are called \emph{triangular identities}. The
functor $F$ is a \emph{left adjoint} for the functor $G$, while
$G$ is a \emph{right adjoint} for the functor $F$.

That $\Sigma\!:\textbf{P}\str \textbf{Qtp}$ has a right adjoint is
shown as follows. Let $\rdd$ and $\ldd$ be defined as at the
beginning of Section~\ref{Q}. For $\mathbb{Q}$ a quasigroup, let
$\nabla_3\!:Q^3\times Q^3\to Q^3$ be defined as in
Section~\ref{twist}, i.e.\ for every $\bar{x}=(x_1,x_2,x_3)$ and
$\bar{y}=(y_1,y_2,y_3)$
\[
\bar{x} \nabla_3 \bar{y}=(x_2 \rdd y_3, x_3 \ldd y_1, x_1\cdot
y_2).
\]
That $(Q^3;\nabla_3)$ is a semisymmetric quasigroup follows from
the fact that the structure $(Q^3;\nabla_1,\nabla_2,\nabla_3)$ is
a semisymmetric twisted quasigroup, which is shown in
Section~\ref{twist}. The semisymmetric quasigroup $(Q^3;\nabla_3)$
is the semisymmetrization $\mathbb{Q}^{\gD}$ of $\mathbb{Q}$
defined at the end of Section~\ref{Q} (see (\ref{eq:SSprod})).

Let $\Delta\!:\textbf{Qtp}\str \textbf{P}$ be a functor, which
maps a quasigroup $\mathbb{Q}$ to the semisymmetric quasigroup
$(Q^3;\nabla_3)$. A homotopy $(f_1,f_2,f_3)$ is mapped by $\Delta$
to the product $f_1\times f_2\times f_3$, which is a homomorphism
as it is shown at the end of Section~\ref{twist}. By the
functoriality of product, we have that $\Delta$ preserves
identities and composition, and it is indeed a functor.

\begin{prop}\label{adjunction}
The functor $\Delta$ is a right adjoint for $\Sigma$.
\end{prop}

\begin{proof}
For $\mathbb{P}$ an object of \textbf{P}, let
$\eta_\mathbb{P}\!:P\str P^3$ be a function defined so that for
every $x\in P$, $\eta_\mathbb{P}(x)=(x,x,x)$. Note that
$\eta_\mathbb{P}$ is a homomorphism from $\mathbb{P}$ to
$\Delta\Sigma\mathbb{P}$, i.e.\ an arrow of \textbf{P}, since
\[
\eta_\mathbb{P}(x\cdot y)=(x\cdot y, x\cdot y, x\cdot y)=(x\rdd
y,x\ldd y,x\cdot y)=\eta_\mathbb{P}(x)\nabla_3\eta_\mathbb{P}(y).
\]
(This is why we consider just a restriction $\Sigma$ of a functor
from \textbf{Q} to \textbf{Qtp}. For an object $\mathbb{P}$ of
\textbf{Q}, the function defined as $\eta_\mathbb{P}$ need not be
a homomorphism.)

Let $\eta$ be the family
\[
\{\eta_\mathbb{P}\mid \mathbb{P}\; \mbox{\rm is an object of }
\textbf{P}\}.
\]
This family is a natural transformation from the identity functor
on \textbf{P} to the composition $\Delta\Sigma$ since for every
arrow $f\!:\mathbb{P}\str \mathbb{Q}$ of \textbf{P} and every
$x\in P$, we have
\[
(\eta_\mathbb{Q}\circ f)(x)=(f(x),f(x),f(x))=(\Delta\Sigma f\circ
\eta_\mathbb{P}) (x).
\]

For $\mathbb{Q}$ an object of \textbf{Qtp} and $i\in\{1,2,3\}$,
let $\pi_i\!:Q^3\str Q$ be the $i$th projection. Let
$\varepsilon_\mathbb{Q}$ be the triple $(\pi_1,\pi_2,\pi_3)$,
which is a homotopy from $\Sigma\Delta \mathbb{Q}$ to
$\mathbb{Q}$, since
\[
\pi_1(\bar{x})\cdot \pi_2(\bar{y})=x_1\cdot
y_2=\pi_3(\bar{x}\:\nabla_3\: \bar{y}).
\]
Hence, $\varepsilon_\mathbb{Q}$ is an arrow of \textbf{Qtp}.

Let $\varepsilon$ be the family
\[
\{\varepsilon_\mathbb{Q}\mid \mathbb{Q}\; \mbox{\rm is an object
of } \textbf{Qtp}\}.
\]
This family is a natural transformation from the composition
$\Sigma\Delta$ to the identity functor on \textbf{Qtp} since for
every arrow $(f_1,f_2,f_3)\!:\mathbb{Q}\str \mathbb{R}$ of
\textbf{Qtp} and every $\bar{x}$ from $\Sigma\Delta \mathbb{Q}$,
we have
\[
(\pi_i\circ (f_1\times f_2\times
f_3))(\bar{x})=f_i(x_i)=(f_i\circ\pi_i)(\bar{x}).
\]

Eventually, we have to show that the following triangular
identities hold for every object $\mathbb{Q}$ of \textbf{Qtp} and
every object $\mathbb{P}$ of \textbf{P}
\[
\Delta(\varepsilon_\mathbb{Q})\circ \eta_{\Delta
\mathbb{Q}}=\mj_{\Delta \mathbb{Q}} \quad\mbox{\rm and}\quad
\varepsilon_{\Sigma \mathbb{P}}\circ
\Sigma(\eta_\mathbb{P})=\mj_{\Sigma \mathbb{P}}.
\]
\begin{tabbing}
For every $\bar{x}\in Q^3$, we have
\\[1.5ex]
\hspace{1.5em}\= $(\Delta(\varepsilon_\mathbb{Q})\circ
\eta_{\Delta \mathbb{Q}})(\bar{x})$ \=
$=\Delta(\varepsilon_\mathbb{Q})(\bar{x},\bar{x},\bar{x})=(\pi_1\times\pi_2\times\pi_3)
(\bar{x},\bar{x},\bar{x})=(x_1,x_2,x_3)$
\\[1ex]
\>\> $=\bar{x}$.
\\[2ex]
For every $x\in P$ and every $i\in\{1,2,3\}$, we have
\\[1.5ex]
\> $(\pi_i\circ \eta_\mathbb{P})(x)=\pi_i(x,x,x)=x$.
\end{tabbing}

\vspace{-4.5ex}
\end{proof}

Moreover, every component of the counit of this adjunction is epi
(i.e.\ right cancellable) and the semisymmetrization is one-one.
This is sufficent for \textbf{Qtp} to be isomorphic to a
subcategory of \textbf{P}. This is one way how to establish this
fact using the previous proposition. It is not clear to us how
this fact is obtained as a corollary of the corresponding
proposition in \cite{Smith}. However, if the goal was just to
establish that \textbf{Qtp} is isomorphic to a subcategory of
\textbf{P}, this adjunction is not necessary at all, which will be
also shown.

A functor $F\!:\textbf{C}\str \textbf{D}$ is \emph{faithful} when
for every pair $f,g\!:A\vdash B$ of arrows of \textbf{C}, $Ff=Fg$
implies $f=g$. An arrow $f\!:A\str B$ of \textbf{C} is \emph{epi}
when for every pair $g,h\!:B\str C$ of arrows of \textbf{C}, the
equality $g\circ f=h\circ f$ implies $g=h$. The following lemmata
will help us to prove that \textbf{Qtp} is isomorphic to a
subcategory of \textbf{P}.

\begin{lem}\label{Dfaith}
The functor $\Delta$ is faithful.
\end{lem}

\begin{proof}
By \cite[IV.3, Theorem~1, Part~(i)]{ML71} (see also
\cite[Section~4, Proposition~4.1]{D14} for an elegant proof of a
related result) it suffices to prove that for every object
$\mathbb{Q}$ of \textbf{Qtp}, the arrow $\varepsilon_\mathbb{Q}$
is epi.

Let $g,h\!:\mathbb{Q}\str \mathbb{R}$ be a pair of arrows of
\textbf{Qtp} such that $ g\circ \varepsilon_\mathbb{Q}=h\circ
\varepsilon_\mathbb{Q}$. This means that for every $i\in\{1,2,3\}$
we have that $g_i\circ\pi_i=h_i\circ \pi_i$. Hence, the function
$g_i$ is equal to the function $h_i$, since the function $\pi_i$
is right cancellable. (However, the homotopy
$\varepsilon_\mathbb{Q}$ need not have a right inverse in
\textbf{Qtp}.)

\vspace{1ex}

The second proof of this lemma is direct and does not rely on
Proposition~\ref{adjunction}. Simply, for homotopies
$(f_1,f_2,f_3)$ and $(g_1,g_2,g_3)$ from $\mathbb{Q}$ to
$\mathbb{R}$, if $f_1\times f_2\times f_3$ and $g_1\times
g_2\times g_3$ are equal as homomorphisms from $\Delta \mathbb{Q}$
to $\Delta \mathbb{R}$ in \textbf{P}, then for every
$i\in\{1,2,3\}$, $f_i=g_i$. Hence, these homotopies are equal
in~\textbf{Qtp}.
\end{proof}

\begin{lem}\label{dif}
If $(Q; \cdot, \rd, \ld)$ and $(Q; \cdot', \rd', \ld')$ are two
different quasigroups, then there are $x,y\in Q$ such that
\[
x\cdot y\neq x\cdot' y.\]
\end{lem}

\begin{proof}
Suppose that for every $x,y\in Q$, $x\cdot y= x\cdot' y$ holds.
Then for every $z,t\in Q$ we have
\[
z\rd t=((z\rd t)\cdot' t)\rd' t=((z\rd t)\cdot t)\rd' t)=z\rd' t.
\]
Analogously, we prove that for every $u,v\in Q$, $u\ld v=u\ld' v$.
Hence, $(Q; \cdot, \rd, \ld)$ and $(Q; \cdot', \rd', \ld')$ are
the same, which contradicts the assumption.
\end{proof}

\begin{lem}
The functor $\Delta$ is one-one on objects.
\end{lem}

\begin{proof}
Suppose that $(Q; \cdot, \rd, \ld)$ and $(Q'; \cdot', \rd', \ld')$
are two different quasigroups. If $Q$ and $Q'$ are different sets,
then $\Delta \mathbb{Q}$ and $\Delta \mathbb{Q}'$ are different.
If $Q=Q'$, then, by Lemma~\ref{dif}, there are $x$ and $y$ in this
set such that $x\cdot y\neq x\cdot' y$. Hence, the operations
$\nabla_3$ for $\Delta \mathbb{Q}$ and $\Delta \mathbb{Q}'$ differ
when applied to $(x,x,x)$ and $(y,y,y)$.
\end{proof}

\vspace{2ex}

As a corollary of these two lemmata we have the following result.

\begin{prop}\label{embedding}
The category {\em{\textbf{Qtp}}} is isomorphic to a subcategory of
{\em{\textbf{P}}}; namely, to its image under the functor
$\Delta$.
\end{prop}

As we have shown by the second proof of Lemma~\ref{Dfaith},
Proposition~\ref{embedding} is independent of
Proposition~\ref{adjunction}. The adjunction, together with this
embedding of \textbf{Qtp} in \textbf{P}, says that the category
\textbf{P} reflects in \textbf{Qtp} in the following sense. A
subcategory \textbf{A} of \textbf{B} is \emph{reflective} in
\textbf{B}, when the inclusion functor from \textbf{A} to
\textbf{B} has a left adjoint called a \emph{reflector}. The
adjunction is called a \emph{reflection} of \textbf{B} in
\textbf{A}.

Propositions \ref{adjunction} and \ref{embedding} say that
\textbf{Qtp} may be considered as a reflective subcategory of
\textbf{P}. The functor $\Sigma$ is a reflector and the adjunction
between $\Sigma$ and $\Delta$ is a reflection of \textbf{P} in
\textbf{Qtp}. However, this does not mean that two quasigroups are
isotopic in \textbf{Qtp} if and only if their semisymmetrizations
are isomorphic in \textbf{P}, which one may conclude from
\cite[first paragraph in the introduction]{ImKoSm}. The reader
should be aware of this potential missusing of these results.

\section{Monadicity of $\Delta$}

For $F\!:\textbf{C}\str \textbf{D}$ a left adjoint for
$G\!:\textbf{D}\str \textbf{C}$, and $\eta$ and $\varepsilon$, the
unit and counit of this adjunction, a $GF$-algebra is a pair
$(C,h)$, where $C$ is an object of \textbf{C} and $h\!: GFC\str C$
is an arrow of \textbf{C} such that the following equalities hold.
\[
h\circ GF h=h\circ G\varepsilon_{FC},\quad\quad h\circ
\eta_C=\mj_C.
\]
A morphism of $GF$-algebras $(C,h)$ and $(C',h')$ is given by an
arrow $f\!:C\str C'$ of \textbf{C} such that $f\circ h=h'\circ
GFf$.

The category $\textbf{C}^{GF}$ has $GF$-algebras as objects and
morphisms of $GF$-algebras as arrows. The \emph{comparison
functor} $K\!:\textbf{D}\str \textbf{C}^{GF}$ is given by
\[
KD=(GD,G\varepsilon_D),\quad\quad Kf=Gf.
\]

In many cases the comparison functor is an isomorphism or an
equivalence (i.e.\ there is a functor from $\textbf{C}^{GF}$ to
\textbf{D} such that both compositions with $K$ are naturally
isomorphic to the identity functors). The right adjoint of an
adjunction or an adjunction are called \emph{monadic} when the
comparison functor is an isomorphism (see \cite[VI.3]{ML71}, also
\cite[Section~4.2]{SR99}). Some other authors (see
\cite[Section~3.3]{BW85}) call an adjunction monadic (tripleable)
when $K$ is just an equivalence.

In the case of adjoint situation involving $\Sigma$ and $\Delta$,
the comparison functor $K\!: \textbf{Qtp}\str
\textbf{P}^{\Delta\Sigma}$ is just an equivalence. To prove this,
by \cite[IV.4, Theorem~1]{ML71} it suffices to prove that $K$ is
full and faithful, and that every $GF$-algebra is isomorphic to
$K\mathbb{Q}$ for some quasigroup $\mathbb{Q}$. The faithfulness
of $K$ follows from \ref{Dfaith} since the arrow function $K$
coincides with the arrow function $\Delta$. That every
$GF$-algebra is isomorphic to $K\mathbb{Q}$ for some quasigroup
$\mathbb{Q}$ is proven in \cite[Section 10, Theorem~33]{Smith8}.

A functor $F\!:\textbf{C}\str \textbf{D}$ is \emph{full} when for
every pair of objects $C_1$ and $C_2$ of $\textbf{C}$ and every
arrow $g\!:FC_1\str FC_2$ of $\textbf{D}$ there is an arrow
$f\!:C_1\str C_2$ of $\textbf{C}$ such that $g=Ff$. It remains to
prove that $K$ is full. For this we use the following lemma.

\begin{lem}\label{lema1}
Every arrow of $\textbf{P}^{\Delta\Sigma}$ from $K\mathbb{Q}$ to
$K\mathbb{R}$ is of the form $f_1\times f_2\times f_3$, for
$(f_1,f_2,f_3)$ a homotopy from $\mathbb{Q}$ to $\mathbb{R}$.
\end{lem}

\begin{proof}
For quasigroups $\mathbb{Q}$ and $\mathbb{R}$ we have that
$K\mathbb{Q}=(\Delta \mathbb{Q},\pi_1\times\pi_2\times\pi_3)$ and
$K\mathbb{R}=(\Delta \mathbb{R},\pi_1\times\pi_2\times\pi_3)$. So,
let
\[
f\!:(\Delta \mathbb{Q},\pi_1\times\pi_2\times\pi_3)\str (\Delta
\mathbb{R},\pi_1\times\pi_2\times\pi_3)
\]
be an arrow of $\textbf{P}^{\Delta\Sigma}$. Since $f$ is a
morphism of $\Delta\Sigma$-algebras, we have that
\[
f\circ(\pi_1\times\pi_2\times\pi_3)=
(\pi_1\times\pi_2\times\pi_3)\circ (f\times f\times f)
\]
as functions from $(Q^3)^3$ to $R^3$.

For $i\in\{1,2,3\}$ and $u\in Q$, let $f_i(u)=\pi_i(f(u,u,u))$.
Let $(x,y,z)$ be an arbitrary element of $Q^3$. Apply the both
sides of the above equality to $((x , x , x ), ( y , y , y ), ( z
, z , z ))\in (Q^3)^3$ in order to obtain
\begin{tabbing}
\hspace{1.5em}\= $f(x,y,z)$ \= $=(\pi_1(f(x,x,x)),
\pi_2(f(y,y,y)), \pi_3(f(z,z,z)))$
\\[1ex]
\>\> $=(f_1(x),f_2(y),f_3(z))$.
\end{tabbing}
Hence, $f=f_1\times f_2\times f_3$ and since it is a homomorphism
from $\Delta \mathbb{Q}$ to $\Delta \mathbb{R}$, we have for every
$\bar{x},\bar{y}\in Q^3$
\[
(f_1\times f_2\times f_3) (\bar{x})\;\nabla_3\; (f_1\times
f_2\times f_3) (\bar{y})=(f_1\times f_2\times
f_3)(\bar{x}\;\nabla_3\; \bar{y}).
\]
By restricting this equality to the third component, we obtain
$f_1(x_1)\cdot f_2(y_2)=f_3(x_1\cdot y_2)$, and hence
$(f_1,f_2,f_3)$ is a homotopy from $\mathbb{Q}$ to $\mathbb{R}$.
\end{proof}


\section{A new semisymmetrization}\label{new}

In this section we introduce a new semisymmetrization functor from
\textbf{Qtp} to \textbf{P}. This leads to another subcategory of
\textbf{P} isomorphic to \textbf{Qtp}. We start with an auxiliary
result.

\begin{lem}\label{homotopy}
The third component $f_3$ of a homotopy is determined by the first
two components $f_1$ and $f_2$.
\end{lem}

\begin{proof}
Let $\mathbb{Q}$ be a quasigroup. For every element $x\in Q$ there
are $y,z\in Q$ such that $x=y\cdot z$ (e.g.\ $x=y\cdot(y\ld x)$).
Hence, $f_3(x)=f_1(y)\cdot f_2(z)$.
\end{proof}

Let $\Gamma\!:\textbf{Qtp}\str \textbf{P}$ be a functor defined on
objects so that $\Gamma \mathbb{Q}$ is a semisymmetric quasigroup
$(Q^2;\nabla)$ (see Section~\ref{biq}) whose elements are pairs
$(x_1,x_2)$, abbreviated by $\hat{x}$, and $\nabla$ is defined so
that
\[
(x_1,x_2)\nabla(y_1,y_2)=(x_2\rdd(x_1\cdot y_2),(x_1\cdot y_2)\ldd
y_1).
\]
(It is straightforward to check that $(\hat{y}\nabla
\hat{x})\nabla \hat{y}=\hat{y}\nabla(\hat{x}\nabla
\hat{y})=\hat{x}$, hence $\Gamma \mathbb{Q}$ is a semisymmetric
quasigroup.)

A homotopy $(f_1,f_2,f_3)$ is mapped by $\Gamma$ to the product
$f_1\times f_2$, which is a homomorphism:
\begin{tabbing}
\hspace{1.5em}$(f_1\times f_2) (\hat{x})\;\nabla\; (f_1\times f_2)
(\hat{y})=$
\\[1ex]
\hspace{11em}\=$=(f_2(x_2)\rdd(f_1(x_1)\cdot
f_2(y_2)),(f_1(x_1)\cdot f_2(y_2))\ldd f_1(y_1))$
\\[1ex]
\>$=(f_1(x_2\rdd(x_1\cdot y_2)),f_2((x_1\cdot y_2)\ldd y_1))$
\\[1ex]
\>$=(f_1\times f_2)(\hat{x}\nabla \hat{y})$.
\end{tabbing}
By the functoriality of product, we have that $\Gamma$ preserves
identities and composition, and it is indeed a functor.

The functor $\Gamma$ is not a right adjoint for $\Sigma$ since a
right adjoint is unique up to isomorphism and $\Gamma \mathbb{Q}$
is not isomorphic to $\Delta \mathbb{Q}$ for every object
$\mathbb{Q}$ of \textbf{Qtp}. However, this adjunction is not
necessary for the faithfulness of $\Gamma$.

\begin{lem}
The functor $\Gamma$ is faithful.
\end{lem}

\begin{proof}
We proceed as in the second proof of Lemma~\ref{Dfaith}. If
$(f_1,f_2,f_3)$ and $(g_1,g_2,g_3)$ are two homotopies from
$\mathbb{Q}$ to $\mathbb{R}$, then
$\Gamma(f_1,f_2,f_3)=\Gamma(g_1,g_2,g_3)$ means that $f_1\times
f_2=g_1\times g_2$. Hence, $f_1=g_1$ and $f_2=g_2$, and by
Lemma~\ref{homotopy}, $f_3=g_3$.
\end{proof}

The functor $\Gamma$, as defined, is not one-one on objects. For
example,
\[
(\{0,1\},+,+,+)\quad \mbox{\rm and} \quad
(\{0,1\},\oplus,\oplus,\oplus),
\]
where $+$ is addition mod 2 and $x \oplus y = x+y+1$\/,
are mapped by $\Gamma$ to
the same object of \textbf{P}. To remedy this matter, one may
redefine $\Gamma$ so that
\[
\Gamma \mathbb{Q}=(Q^2\times\{\mathbb{Q}\},\nabla),
\]
where $\mathbb{Q}$, as the third component of every element,
guarantees that $\Gamma$ is one-one on objects. The operation
$\nabla$ is defined as above, just neglecting the third component.
Hence, \textbf{Qtp} may be considered as another subcategory of
\textbf{P}.

\section*{\centerline{\small Acknowledgements}}

We are grateful to J.D.H.\ Smith for his comments on the previous
version of this note. Work on this paper was supported by the
Ministry of Education and Science of Serbia through grants
ON174008 (first author) and ON174026 (both authors). 
\vspace{1cm}
\noindent
{Mathematical Institute \\
         of the Serbian Academy of Sciences and Arts\\
         Knez Mihailova 36\\
         11001 Beograd \\
         Serbia \\
e-mail: \{sasa,zpetric\}@mi.sanu.ac.rs }

\end{document}